\documentclass[11pt]{amsart}

\title[A probabilistic heuristic\ldots]{A probabilistic heuristic for counting components of functional graphs of polynomials over finite fields}
\author{Elisa Bellah}
\address{
Department of Mathematics, University of Oregon}
\email{\href{mailto:ebellah@uoregon.edu}{ebellah@uoregon.edu}}
\author{Derek Garton}
\author{Erin Tannenbaum}
\author{Noah Walton}
\address{Fariborz Maseeh Department of Mathematics and Statistics, Portland State University}
\email{\href{mailto:gartondw@pdx.edu}{gartondw@pdx.edu},
\href{mailto:ejt3@pdx.edu}{ejt3@pdx.edu},
\href{mailto:nwalton@pdx.edu}{nwalton@pdx.edu}}
\date{\today}

\usepackage{amssymb,
url,
mathtools,
siunitx}

\usepackage[breaklinks=true]{hyperref}
\hypersetup{colorlinks=true,
urlcolor=blue,
linkcolor=cyan
}

\usepackage{graphicx}

\usepackage[top=1in,bottom=1in,left=1in,right=1in]{geometry}

\newcommand{\Z}{\ensuremath{{\mathbb{Z}}}}

\renewcommand{\P}{\ensuremath{{\mathbb{P}}}}

\newcommand{\F}{\ensuremath{{\mathbb{F}}}}

\newcommand{\lv}{\ensuremath{\left\vert}}
\newcommand{\rv}{\ensuremath{\right\vert}}
\newcommand{\lp}{\ensuremath{\left(}}
\newcommand{\rp}{\ensuremath{\right)}}
\newcommand{\lb}{\ensuremath{\left\{}}
\newcommand{\rb}{\ensuremath{\right\}}}
\newcommand{\lc}{\ensuremath{\left[}}
\newcommand{\rc}{\ensuremath{\right]}}
\newcommand{\xqk}{\ensuremath{X_{q,k}}}
\newcommand{\yqdk}{\ensuremath{Y_{q, d, k}}}
\newcommand{\cqk}{\ensuremath{\mathfrak{C}(q, k)}}
\newcommand{\tqk}{\ensuremath{\mathfrak{T}(q,k)}}

\DeclareMathOperator*{\E}{\mathbb{E}}

\theoremstyle{remark}
\newtheorem{bigquestion}{Question}[section]
\newtheorem{obs}[bigquestion]{Observation}
\newtheorem{cqksize}[bigquestion]{Remark}
\newtheorem{xexpected}[bigquestion]{Remark}
\newtheorem{wrong}[bigquestion]{Remark}
\newtheorem{admission}[bigquestion]{Remark}
\newtheorem{constantsize}[bigquestion]{Remark}

\theoremstyle{plain}
\newtheorem{yexpected}[bigquestion]{Proposition}
\newtheorem{heuristic}[bigquestion]{Heuristic}
\newtheorem{conjecture1}[bigquestion]{Conjecture}
\newtheorem{mainthm1}[bigquestion]{Theorem}
\newtheorem{mainthm2}[bigquestion]{Theorem}

\newtheorem{ratheuristic}[bigquestion]{Heuristic}
\newtheorem{conjecture2}[bigquestion]{Conjecture}

\theoremstyle{definition}

\begin{document}

\maketitle

\begin{abstract}
In 2014, Flynn and Garton~\cite{FG} bounded the average number of components of the functional graphs of polynomials of fixed degree over a finite field.
When the fixed degree was large (relative to the size of the finite field), their lower bound matched Kruskal's asymptotic for random functional graphs.
However, when the fixed degree was small, they were unable to match Krusal's bound, since they could not (Lagrange) interpolate cycles in functional graphs of length greater than the fixed degree.
In our work, we introduce a heuristic for approximating the average number of such cycles of any length.
This heuristic is, roughly, that for sets of edges in a functional graph, the quality of being a cycle and the quality of being interpolable are ``uncorrelated enough''.
We prove that this heuristic implies that the average number of components of the functional graphs of polynomials of fixed degree over a finite field is within a bounded constant of Kruskal's bound.
We also analyze some numerical data comparing implications of this heuristic to some component counts of functional graphs of polynomials over finite fields.
\end{abstract}


\section{Introduction}
\label{intro}

A \emph{\textup{(}discrete\textup{)} dynamical system} is a pair $\left(S,f\right)$ consisting of a set $S$ and a map $f:S\to S$.
Given such a system, an element $s\in S$ is a \emph{periodic point} of the system if there exists some $k\in\Z^{>0}$ such that $(\overbrace{f\circ\cdots\circ f}^{k\text{ times}})(s)=s$; the smallest $k\in\Z^{>0}$ with this property is called the \emph{period} of $s$.
The \emph{functional graph} of such a system, which we denote by $\Gamma(S,f)$, is the directed graph whose vertex set is $S$ and whose edges are given by the relation $s\to t$ if and only if $f(s)=t$.
A \emph{component} of such a graph is a component of the underlying undirected graph.
For any $n\in\Z^{>0}$, let $\mathcal{K}(n)$ denote the average number of components of a random functional graph on a set of size $n$; that is, choose any set $S$ with $\lv S\rv=n$ and let
\[
\mathcal{K}(n)
=n^{-n}
\sum_{f:S\to S}{\lv\lb\text{components of }\Gamma(S,f)\rb\rv}.
\]
Kruskal (see~\cite{K}) proved that
\[
\mathcal{K}(n)
=\frac{1}{2}\log{n}+\left(\frac{\log{2}+C}{2}\right)+o(1),
\]
where $C=.5772\ldots$ is Euler's constant.

Recently, researchers have begun studying the analogous situation for polynomials and rational maps over finite fields.
More precisely, if $q$ is a prime power, define $\Gamma(q,f)=\Gamma\lp\F_q,f\rp$ if $f\in\F_q[x]$ and $\Gamma(q,f)=\Gamma\lp\P^1(\F_q),f\rp$ if $f\in\F_q(x)$.
(If there is no ambiguity, we will frequently write $\Gamma_f$ for $\Gamma(q,f)$.)
Then we can ask the question: for $d\in\Z^{>0}$, what is the average number of components of $\Gamma_f$, for $f$ ranging over all polynomials (or rational maps) over $\F_q$ of a fixed degree?
For example, if we define
\[
\mathcal{P}(q,d)
:=\frac{1}{\lv\lb f\in\F_q[x]\mid\deg{f}=d\rb\rv}\cdot
\sum_{\substack{f\in\F_q[x]\\\deg{f}=d}}
{\lv\lb\text{components of }\Gamma_f\rb\rv},
\]
then we can ask:
\begin{bigquestion}\label{bigquestion}
For a prime power $q$ and $d\in\Z^{>0}$, how does $\mathcal{P}(q,d)$ compare to $\mathcal{K}(q)$?
\end{bigquestion}

In this paper, we recast these questions in probabilistic terms.
Specifically, in \hyperref[families]{Section~\ref*{families}}, we define two families of random variables whose interaction determines the answer to \hyperref[bigquestion]{Question~\ref*{bigquestion}}.
(Briefly, both families random variables have sample space a certain collection of subsets of $\F_q\times\F_q$---one random variable determines if a collection is a cycle, and the other returns how many polynomials of a given degree pass though every point in a collection.)

Our main result, \hyperref[mainthm1]{Theorem~\ref*{mainthm1}}, states that if these two familes of random variables satisfy a certain ``noncorrelation hypothesis'', then
\[
\mathcal{P}(q,d)=\mathcal{K}(q)+O(1).
\]
(See \hyperref[heuristic]{Heuristic~\ref*{heuristic}} for an exact formulation of this hypothesis.)
In \hyperref[families]{Section~\ref*{families}} we define and study these random variables; in particular, we compute their expected values.
Next, in \hyperref[assumption]{Section~\ref*{assumption}} we use the results from \hyperref[families]{Section~\ref*{families}} to prove  the aforementioned \hyperref[mainthm1]{Theorem~\ref*{mainthm1}}.
Then, in \hyperref[evidence]{Section~\ref*{evidence}}, we provide numerical evidence in support of \hyperref[heuristic]{Heuristic~\ref*{heuristic}}.
Finally, these results carry over easily to the analogous question for rational functions; these results make up \hyperref[rationalfunctions]{Section~\ref*{rationalfunctions}}.

Previous work of Flynn and the second author (see~\cite{FG}) provided a partial answer to the question under discussion.
In particular, they proved that if $d\geq\sqrt{q}$, then the average number of components of functional graphs of polynomials (or rational maps) of degree $d$ over $\mathbb{F}_q$ is bounded below by
\[
\frac{1}{2}\log{q}-4
\]
(this is Corollary~2.3 and Theorem~3.6 from~\cite{FG}).

To describe their method, which is the starting point for this paper, we require a definition and an observation.
If a map $f$ has a periodic point $s$ of period $k$, with orbit $s=s_1\xrightarrow{f}\cdots\xrightarrow{f}s_k\xrightarrow{f}s_1$, then we refer to its orbit as a \emph{cycle} (cycles of length $k$ are called \emph{$k$-cycles}).
(See~\cite{VS} for more exposition and illustrations of the cycle structure of functional graphs.)
This definition is especially useful since it allows for the following observation.
\begin{obs} \label{obs}
Components of $\Gamma_f$ are in one-to-one correspondence with the cycles of $f$.
\end{obs}
\noindent To obtain their results, Flynn and the second author used Lagrange interpolation to interpolate all the cycles of length smaller than the degree of the maps in question.
Since they could not interpolate longer cycles,
\begin{itemize}
\item they obtained only a lower bound for $\mathcal{P}(q,d)$, and
\item their result required that $d$ be at least $\sqrt{q}$.
\end{itemize}
See \hyperref[admission]{Remark~\ref*{admission}} for a discussion on the relationship between the results of this paper and the results of~\cite{FG}; for example, they proved that the random variables mentioned above are indeed uncorrelated in certain cases.

The cycle structure of functional graphs of polynomials over finite fields has been studied extensively in certain cases.
Vasiga and Shallit~\cite{VS} studied the cycle structure of $\Gamma_f$ for the cases $f=x^2$ and $f=x^2-2$, as did Rogers~\cite{MR1368298} for $f=x^2$.
For any $m\in\Z^{>0}$, the squaring function is also defined over $\Z/m\Z$; Carlip and Mincheva~\cite{MR2402530} addressed this situation for certain $m$.
Similarly, Chou and Shparlinski~\cite{MR2072394} studied the cycle structure of repeated exponentiation over finite fields of prime size.
In the context of Pollard's rho algorithm for factoring integers (see~\cite{MR0392798}), researchers have provided copious data and heuristic arguments supporting the claim that quadratic polynomials produce as many ``collisions'' as random functions, but very little has been proven (see~\cite{MR0392798} and~\cite{MR1094034}).
For many other aspects of functional graphs besides their cycle structure, see~\cite{FO} for a study of about twenty characteristic parameters of random mappings in various settings.

More recently, Burnette and Schmutz~\cite{BS} used the probabilistic point of view to study a similar question to the one we address here.
If $f$ is a polynomial (or rational function) over $\F_q$, define the \emph{ultimate period of $f$} to be the least common multiple of the cycle lengths of $\Gamma_f$.
They found a lower bound for the average ultimate period of polynomials (and rational functions) of fixed degree, whenever the degree of the maps in question, and the size of the finite field, were large enough.

\section{Two families of random variables}
\label{families}

In this section, we define two families of random variables and compute their expected values.
The interaction of these random variables determines the answer to \hyperref[bigquestion]{Question~\ref*{bigquestion}}; see \hyperref[wrong]{Remark~\ref*{wrong}} and the remarks that follow for details about this interaction.
For the remainder of the section, fix a prime power $q$ and positive integer $d$.
Now, for any set $S$ and $C\subseteq S \times S$, we say that $C$ is \emph{consistent} if and only if it has the following property: if $(a,b),(a,c)\in C$, then $b=c$.
Next, for any $k\in\Z^{\geq0}$, define
\[
\cqk
=\lb C\subseteq\F_q\times\F_q
\mid C\text{ is consistent and }\lv C\rv=k\rb.
\]
Any element of $C\in\cqk$ defines a directed graph with vertex set $\F_q$ and edge set $\lb s\to t\mid(s,t)\in C\rb$; let $X_{q,k}:\cqk\to\lb0,1\rb$ be the binary random variable that detects whether or not an element of $\cqk$ defines a graph that happens to be a $k$-cycle.
If $f\in\F_q[x]$ and $C\in\cqk$, we say that $f$ \emph{satisfies} $C$ if $f(a)=b$ for all $(a,b)\in C$.
Next, we let $Y_{q,d,k}:\cqk\to\Z^{\geq0}$ be the random variable defined by
\[
Y_{q,d,k}(C)=\lv\lb f\in\F_q[x]\mid\deg{f}=d\text{ and }f\text{ satisfies }C\rb\rv.
\]

Before computing the expected values of $X_{q,k}$ and $Y_{q,d,k}$, we first mention the size of their sample space.

\begin{cqksize}\label{cqksize}
If $k\in\Z^{>0}$, then
\[
\lv\cqk\rv=q^k \binom{q}{k}.
\]
\end{cqksize}

\begin{proof}
Since the elements of $\cqk$ are consistent, there are $\binom{q}{k}$ possible choices for the sets of abscissas for any choice of ordinates.
Since the ordinates of elements of $\cqk$ are unrestricted, we conclude that $\lv\cqk\rv= \binom{q}{k}q^k$.
\end{proof}

\begin{xexpected}\label{xexpected}
If $k\in\lb1,\ldots,q\rb$, then
\[
\E[\xqk]
=\frac{q(q-1) \cdots (q-(k-1))}{k\lv\cqk\rv}
=\frac{(k-1)!}{q^k}.
\]
\end{xexpected} 
\begin{proof} 
Since
\[
\E[\xqk]
=\frac{\lv\lb C\in\cqk\mid C\text{ is a cycle}\rb\rv}{\lv\cqk\rv},
\]
we only need to count the number of elements in $\cqk$ that are cycles. Since there are
\[
\frac{q(q-1)\cdots (q-(k-1))}{k}
\]
cycles of length $k$, we conclude by \hyperref[cqksize]{Remark~\ref*{cqksize}}.
\end{proof} 

\begin{yexpected}\label{yexpected}
If $k\in\lb1,\ldots,q\rb$, then
\[
\E[\yqdk]=q^{d+1-k}-q^d.
\]
\end{yexpected}                   
\begin{proof}
Since
\begin{align*}
\sum_{C\in\cqk}{\yqdk(C)}
&=\sum_{C\in\cqk}
{\lv\lb f \in \F_q[x]
\mid\deg{f}=d\text{ and }f\text{ satisfies }C\rb\rv}\\
&=\sum_{\substack{f \in\F_q[x]
\\\deg{f}=d}}
{\lv\lb C\in\cqk
\mid C\text{ is satisfied by }f\rb\rv}\\
&=\sum_{\substack{f\in\F_q[x]
\\\deg{f}=d}}
{\binom{q}{k}}\\
&=\lp q^{d+1}-q^d\rp\binom{q}{k},
\end{align*}
we see by \hyperref[cqksize]{Remark~\ref*{cqksize}} that
\[
\E[\yqdk]
=\lv\cqk\rv^{-1}\cdot\sum_{C\in\cqk}{\yqdk(C)}
=\frac{\lp q^{d+1}-q^d\rp\binom{q}{k}}{q^k \binom{q}{k}}=q^{d+1-k}-q^{d-k}.
\]
\end{proof}

\begin{wrong}\label{wrong}
If we assume that $X_{q,d},Y_{q,d,k}$ are uncorrelated for all $k\in\lb1,\ldots,q\rb$, then $\mathcal{K}(q)=\mathcal{P}(q,d)$.
\end{wrong}
\begin{proof}
Note that for any $k\in\lb1,\ldots,q\rb$,
\begin{align*}
\sum_{\substack{f \in \F_q[x] \\ \deg{f}=d}}
{\lv\lb k\text{-cycles in } \Gamma_f\rb\rv}
&= \sum_{C \in \cqk} \xqk \yqdk(C)&&\\
&=\lv\cqk\rv\E[\xqk \yqdk]&&\\
&=\lv\cqk\rv\E[\xqk] \E [\yqdk]&&\text{by assumption}.
\end{align*}
Now we can apply Remarks~\ref{cqksize} and~\ref{xexpected}, along with \hyperref[yexpected]{Proposition~\ref*{yexpected}}, to see that
\begin{align*}
\mathcal{P}(q,d)
&=\frac{\lv\cqk\rv}{q^{d+1}-q^d}\cdot\sum_{k=1}^{q}{\E\lc\xqk\rc\E\lc\yqdk\rc}&&\\
&=\sum_{k=1}^q{\frac{q(q-1) \cdots (q-(k-1))}{k q^k}}&&\\
&=\mathcal{K}(q)&&\text{by (16) in \cite{K}}.
\end{align*}
\end{proof}

\begin{admission}\label{admission}
Unfortunately, we must face up to the fact that the random variables $X_{q,d},Y_{q,d,k}$ are not uncorrelated for all $k\in\lb1,\ldots,q\rb$.
Indeed, if they were, then the computations from \hyperref[wrong]{Remark~\ref*{wrong}} would show that
\[
\sum_{\substack{f\in\F_q[x]\\\deg{f}=2}}
{\lv\lb q\text{-cycles in }\Gamma_f\rb\rv}
=\frac{q!(q-1)}{q^{q-2}}.
\]
But, if $q>3$, then the quantity on the left is an integer, and the quantity on the right is not!
In \hyperref[assumption]{Section~\ref*{assumption}}, we propose a heuristic that is more reasonable than that these two random variables are uncorrelated.

On the other hand, we should note that the variables $X_{q,d},Y_{q,d,k}$ are indeed uncorrelated whenever $k\in\lb1,\ldots,d\rb$; this is the content of Lemma~2.1 in~\cite{FG}.
\end{admission}

\section{The heuristic assumption and its implications}
\label{assumption}

As mentioned in \hyperref[admission]{Remark~\ref*{admission}}, the variables $X_{q,d},Y_{q,d,k}$ are not uncorrelated for all $k\in\lb1,\ldots,q\rb$.
In this section, we propose a weaker heuristic for these variables, one which nevertheless implies $\mathcal{P}\lp q,d\rp=\mathcal{K}(q)+O(1)$.

\begin{heuristic}\label{heuristic}
For any $k \in\Z^{>0}$ and any $d\in\Z^{\geq0}$,
\[
\E \lc \xqk \yqdk \rc
=\E \lc \xqk \rc \E \lc \yqdk \rc +O\lp q^{d-2k} \rp.
\]
Here, the implied constant depends only on $d$.
\end{heuristic}

In fact, \hyperref[heuristic]{Heuristic~\ref*{heuristic}} implies more than $\mathcal{P}\lp q,d\rp=\mathcal{K}(q)+O(1)$; we state the stronger implication here as a conjecture after one more definition.
If $k \in\Z^{>0}$ and any $d\in\Z^{\geq0}$, let
\[
\mathcal{P}(q,d,k)
:=\frac{1}{\lv\lb f\in\F_q[x]\mid\deg{f}=d\rb\rv}\cdot
\sum_{\substack{f\in\F_q[x]\\\deg{f}=d}}
{\lv\lb k\text{-cycles in }\Gamma_f\rb\rv}.
\]

\begin{conjecture1}\label{conjecture1}
For any $k \in\Z^{>0}$ and any $d\in\Z^{\geq0}$,
\[
\mathcal{P}(q, d, k)
=\frac{q(q-1) \cdots (q-(k-1))}{kq^k} + O\lp\frac{1}{q}\rp,
\]
where the implied constant depends only on $d$.
In particular, $\mathcal{P}\lp q,d\rp=\mathcal{K}(q)+O(1)$.
\end{conjecture1}

\begin{mainthm1} \label{mainthm1} If \hyperref[heuristic]{Heuristic ~\ref*{heuristic}} is true, then \hyperref[conjecture1]{Conjecture~\ref*{conjecture1}} is true. \end{mainthm1}

\begin{proof}
As in the proof of \hyperref[wrong]{Remark~\ref*{wrong}}, \hyperref[heuristic]{Heuristic ~\ref*{heuristic}} immediately implies that
\[
\sum_{\substack{f \in \F_q[x] \\ \deg{f}=d}}
{\lv\lb k\text{-cycles in } \Gamma_f\rb\rv}
=\lv\cqk\rv\lp \E[\xqk] \E [\yqdk] + O \lp q^{d-2k} \rp \rp.
\]
Next, we can apply Remarks~\ref{cqksize} and~\ref{xexpected}, along with \hyperref[yexpected]{Proposition~\ref*{yexpected}}, to see that
\begin{align*}
\sum_{\substack{f \in \F_q[x] \\ \deg{f}=d}}
{\lv\lb k\text{-cycles in } \Gamma_f\rb\rv}
&=\frac{q(q-1) \cdots (q-(k-1))}{kq^k}\lp q^{d+1}-q^d\rp+\binom{q}{k} q^k\cdot O \lp q^{d-2k} \rp\\
&=\frac{q(q-1) \cdots (q-(k-1))}{kq^k}\lp q^{d+1}-q^d\rp+O\lp q^d\rp.
\end{align*}
To conclude, note that
\begin{align*}
\mathcal{P}(q,d,k)
&=\frac{1}{{q^{d+1}-q^d}}\cdot
\sum_{\substack{f\in\F_q[x]\\\deg{f}=d}}
{\lv\lb k\text{-cycles in }\Gamma_f\rb\rv}\\
&=\frac{q(q-1) \cdots (q-(k-1))}{k q^k}+ O \lp\frac{1}{q}\rp.
\end{align*}
\end{proof}

\begin{constantsize}\label{constantsize}
The available numerical data suggests that the implied constants in \hyperref[heuristic]{Heurisitic~\ref*{heuristic}} could be quite small.
For example, the constant for $d=2$ seems as if it could be as small as 60.
(See \hyperref[evidence]{Section~\ref*{evidence}} for more details on the available data.)
\end{constantsize}

\section{Numerical evidence}
\label{evidence}

In constructing numerical evidence for \hyperref[conjecture1]{Conjecture~\ref*{conjecture1}}, we computed the number of cycles of every length for all polynomials in $\F_q[x]$
\begin{itemize}
\item
of degree 2, up to $q=241$, and 
\item
of degree 3 up to $q=73$.
\end{itemize}
For the remainder of the section, we will address only the quadratic case; a similar analysis works for the cubic case.

Of course, if we let $\mathfrak{Q}=\lb q\in\Z\mid q\text{ is a prime power, and }2\leq q\leq241\rb$, then for any $k\in\lb1,\ldots,241\rb$, there is certainly a constant---let's call it $C_k$---for which
\[
\lv\mathcal{P}\lp q,2,k\rp-\frac{q(q-1)\cdots\lp q-(k-1)\rp}{kq^k}\rv\leq C_k\cdot\frac{1}{q}
\hspace{10px}\text{for all}\hspace{10px}
q\in\mathfrak{Q}.
\]
There are two obvious questions to ask about these constants, which we will address in turn
\begin{itemize}
\item
For any particular $k$, how plausible is it that $\lv\mathcal{P}\lp q,2,k\rp-\frac{q(q-1)\cdots\lp q-(k-1)\rp}{kq^k}\rv\leq C_k\cdot\frac{1}{q}$ for \emph{all} prime powers $q$?
\item
Even if $\mathcal{P}(q,2,k)=\frac{q(q-1)\cdots\lp q-(k-1)\rp}{kq^k}+O\lp\frac{1}{q}\rp$ for all $k\in\Z^{>0}$, does it seem likely that the implied constants are bounded, as asserted by \hyperref[conjecture1]{Conjecture~\ref*{conjecture1}}?
\end{itemize}

To answer the former question, we could plot, for various $k$,
\[
\mathcal{P}(q,2,k)
\hspace{10px}\text{and}\hspace{10px}
\frac{q(q-1)\cdots(q-(k-1))}{kq^k}\pm C_k\cdot\frac{1}{q}.
\]
But, as these numbers quickly become minuscule, it is convenient to let
\[
\widehat{\mathcal{P}}(q,d,k)
=\lv\lb f\in\F_q[x]\mid\deg{f}=d\rb\rv\cdot\mathcal{P}(q,d,k)
=\lp q^{d+1}-q^d\rp\cdot\mathcal{P}(q,d,k);
\]
that is, $\widehat{\mathcal{P}}(q,d,k)$ is the number of $k$-cycles appearing in functional graphs of polynomials in $\F_q[x]$ of degree $d$.
\hyperref[conjecture1]{Conjecture~\ref*{conjecture1}} predicts that this quantity is about
\[
\lp q^{d+1}-q^d\rp\cdot\frac{q(q-1)\cdots(q-(k-1))}{kq^k},
\]
which we will denote by $\mathcal{G}\lp q,d,k\rp$.
By the definition of $C_k$, we know that for all $q\in\mathfrak{Q}$ and $k\in{1,2,\ldots,241}$,
\[
\lv\widehat{\mathcal{P}}(q,2,k)
-\mathcal{G}(q,2,k)\rv
\leq C_k\lp q^2-q\rp.
\]
As two examples of the data we have compiled, we include plots of $\widehat{\mathcal{P}}(q,2,k)$ and $\mathcal{G}(q,2,k)\pm C_k\lp q^2-q\rp$ for $k=6,10$, where $C_6=59$ and $C_{10}=14$.
These graphs are typical for $k\in\lb1,\ldots,241\rb$.

\subsection{Plots of $\widehat{\mathcal{P}}(q,2,k)$ and $\mathcal{G}(q,2,k)\pm C_k\lp q^2-q\rp$ for $k=6,10$}
\label{firstquestion}
~\\
~\\
\includegraphics[width=3in,height=3in,viewport=0 0 270 270]{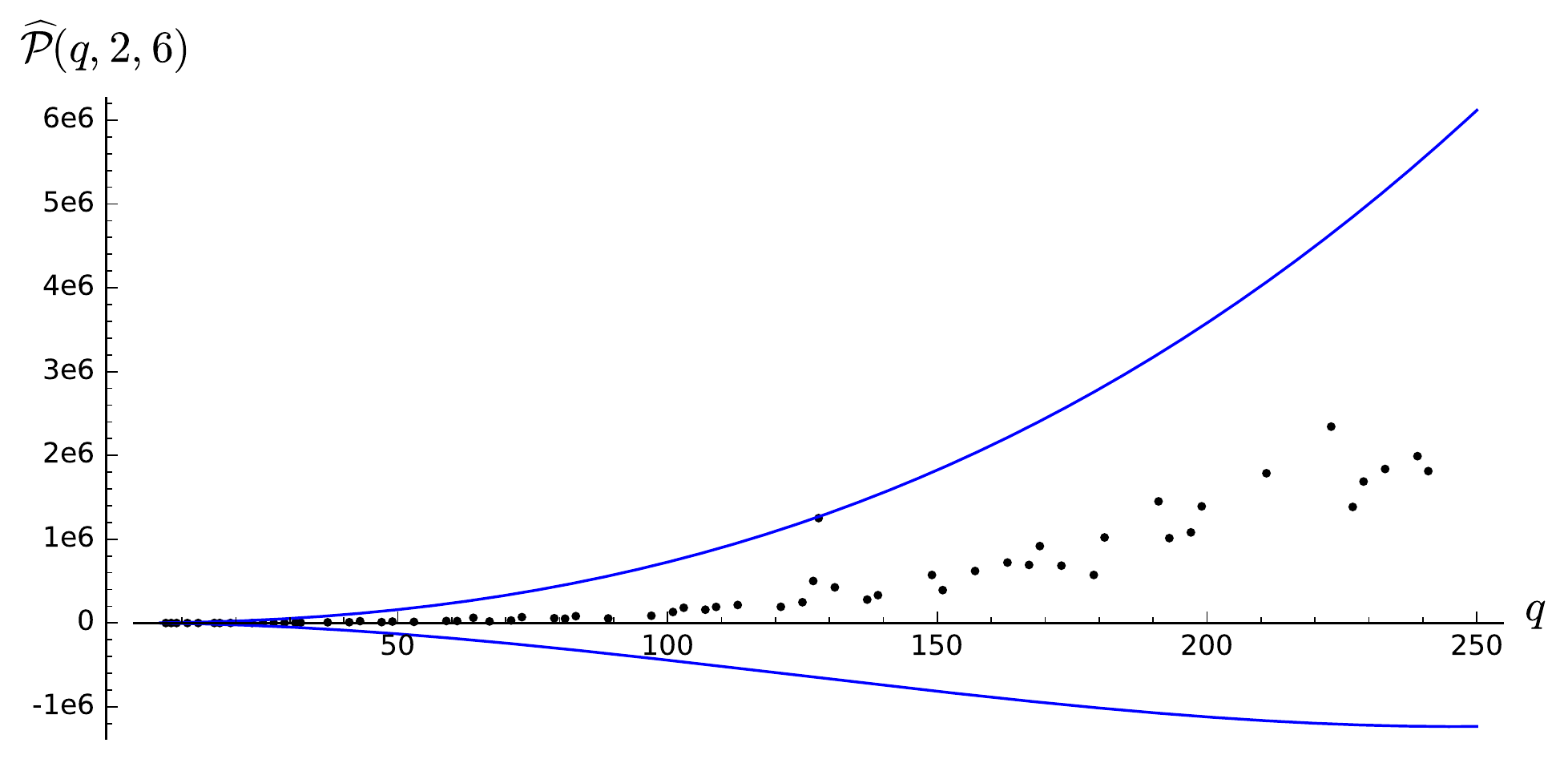}
\vspace{-20px}
\[
C_6\approx59.06
\]

\includegraphics[width=3in,height=3in,viewport=0 0 270 270]{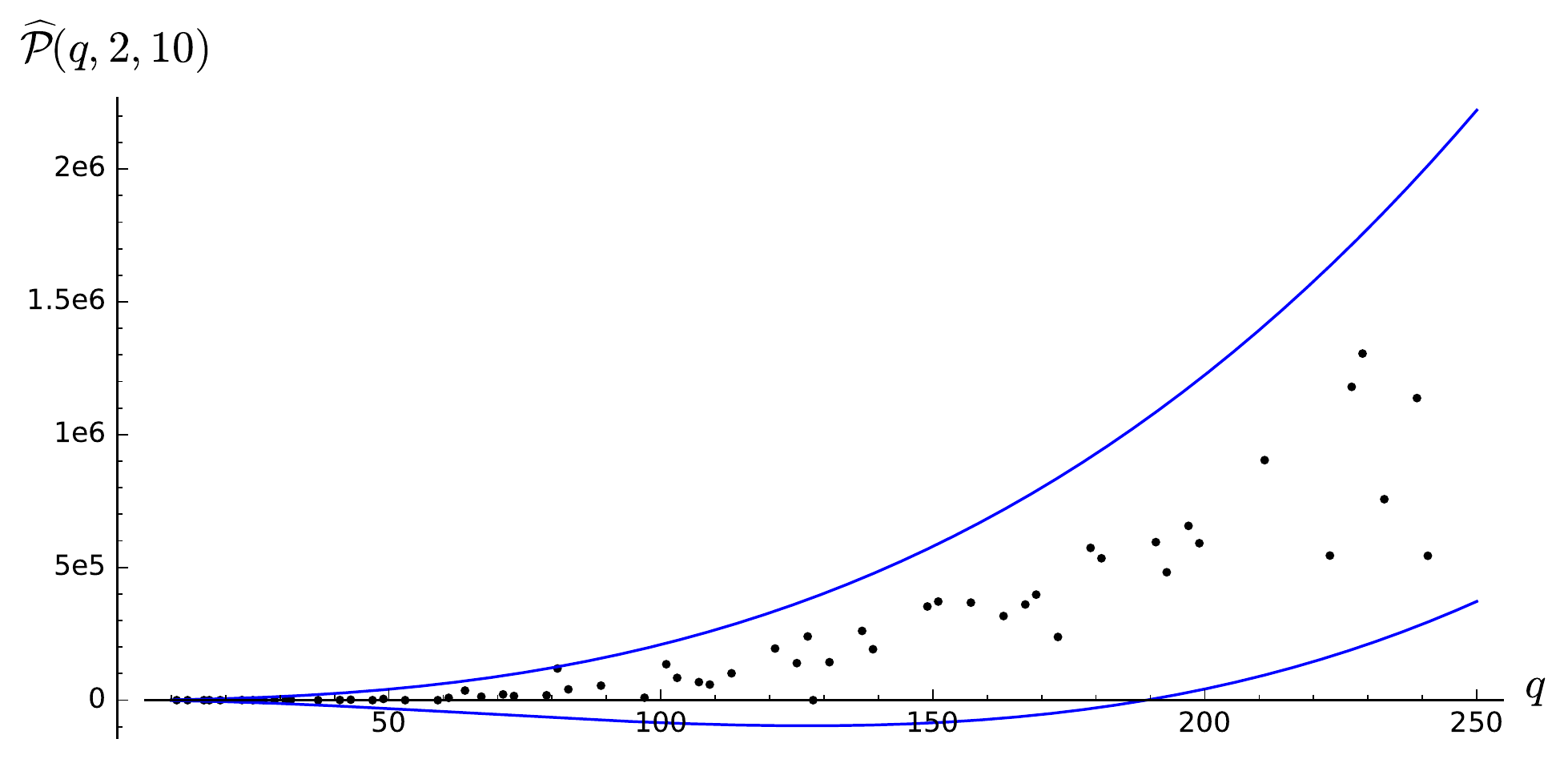}
\vspace{-15px}
\[
C_{10}\approx14.86
\]

\subsection{A plot of $C_k$}
\label{secondquestion}

To address the second question mentioned above, we plot the various values of $C_k$ in the hopes that they appear to be bounded.
This graph is below.\\

\includegraphics[width=3in,height=3in,viewport=0 0 270 270]{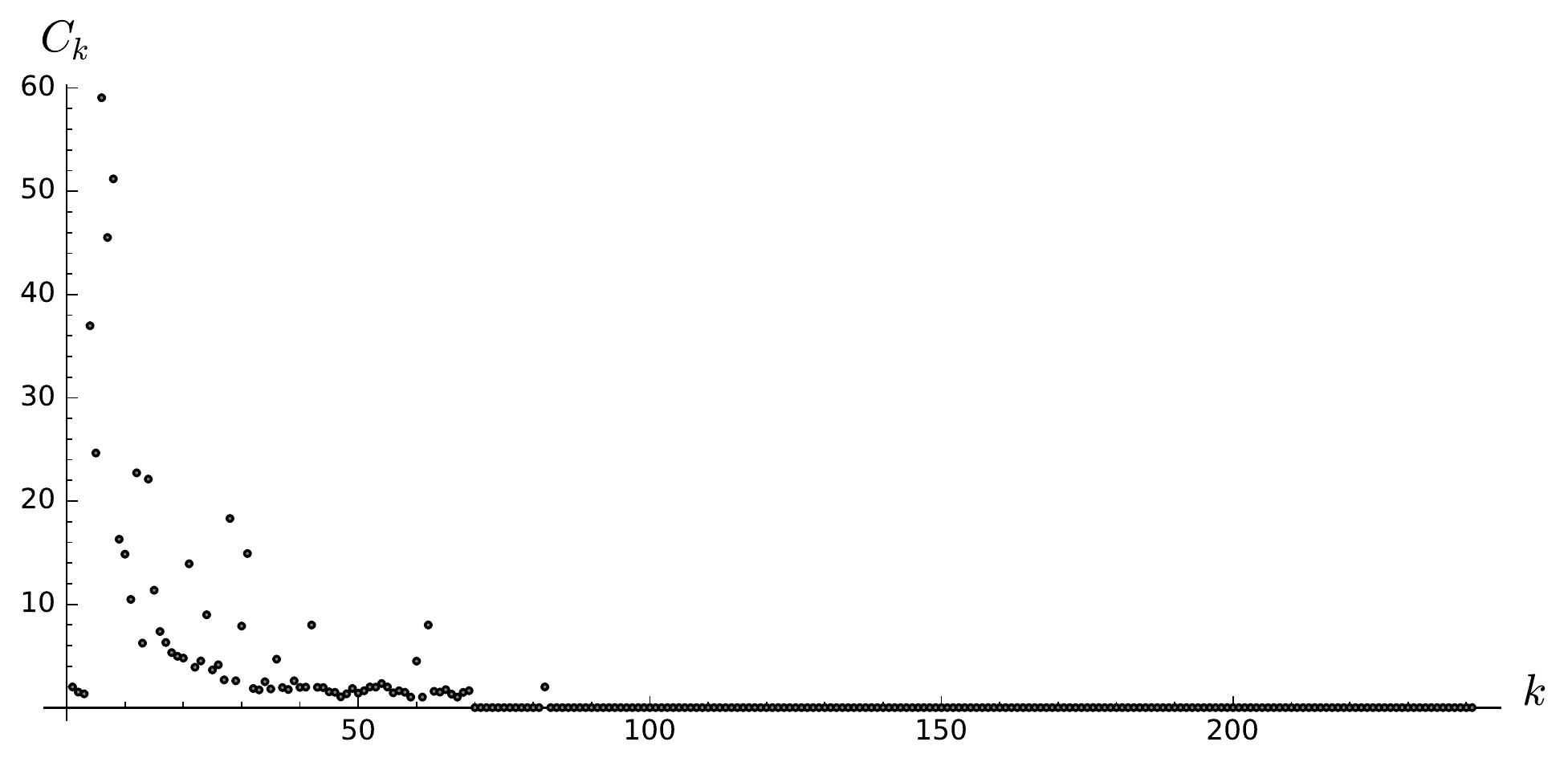}

We should point out that the small values of $C_k$ in the above graph are a result of the fact that in our data, we simply found no $k$-cycles at all for all $k>82$.
So from $k=82$ onward, the above graph is simply plotting
\[
\frac{241!}{(241-k)!\cdot k \cdot 241^{k-1}}.
\]
This begs two questions:
\begin{itemize}
\item
As cycles of larger length arise for larger values of $q$, will the size of $C_k$ increase?
\item
Conversely, if these cycles do not arise promptly, will this increase the size of $C_k$?
\end{itemize}
Of course, we cannot answer these questions, but note that for the particular value of $k=82$, the quadratic polynomials we tested yielded exactly $\num[group-separator={,}]{27722}$ $82$-cycles (all appearing when $q=167$), whereas for $k\in\lb70,\ldots,81\rb$, they yielded exactly zero.
That is, this is an example of a cycle of larger length arising without affecting the maximum of the $C_k$s.

As for the second question, the lack of $k$-cycles will not cause $C_k$ to rise above 60 as long as the first $k$-cycle appears in a graph for a finite field of size less than $60k$.
For example, the smallest $q$ for which 62-cycles appear is $q=128$ (which is well under $60\cdot62$).
The smallest cycle length that does not appear for $q\in\mathfrak{Q}$ is $k=43$; if a 43-cycle does not appear by the time $q=2579$, then $C_{43}$ will rise above 60.
It is unfortunately beyond our abilities to determine if a 43-cycle appears by this time.

\section{Rational functions}
\label{rationalfunctions}

In this section, we briefly mention the results for rational functions, which are analogous to those for polynomials.
For any prime power $q$ and $d\in\Z^{\geq0}$, let
\[
\mathcal{R}(q,d)
:=\frac{1}{\lv\lb f\in\P^1(\F_q)[x]\mid\deg{(f)}=d\rb\rv}\cdot
\sum_{\substack{f\in\P^1(\F_q)[x]\\\deg{(f)}=d}}
{\lv\lb\text{cycles in }\Gamma_f\rb\rv}.
\]
If $k\in\Z^{>0}$, we can define $\mathcal{R}(q,d,k)$ in exactly the same way as $\mathcal{P}(q,d,k)$.

To define our new families of random variables, for any prime power $q$ and $k\in\Z^{>0}$, let
\[
\tqk
=\lb T\subseteq\P^1(\mathbb{F}_q)\times\P^1(\mathbb{F}_q)\mid T\text{ is consistent and }\lv T\rv=k\rb,
\]
 and $V_{q,k}:\tqk\to\lb0,1\rb$ be the binary random variable that detects whether or not an element of $\tqk$ is a $k$-cycle.
If $d\in\Z^{\geq0}$, let $W_{q,d,k}:\tqk\to\Z^{\geq0}$ be the random variable defined by
\[
W_{q,d,k}(T)=\lv\lb f\in\F_q(x)\mid\deg{f}=d\text{ and }f\text{ satisfies }T\rb\rv.
\]

The rational function analogs of \hyperref[cqksize]{Remark~\ref*{cqksize}}, \hyperref[xexpected]{Remark~\ref*{xexpected}}, \hyperref[yexpected]{Proposition~\ref*{yexpected}} are proved as above, leading to the following conjecture, which again follows from the heuristic that the random variables $V_{q,k},W_{q,d,k}$ are ``uncorrelated enough''.



%
%
%
%
%
%
%
\begin{conjecture2}\label{conjecture2}
For any $k \in\Z^{>0}$ and any $d\in\Z^{\geq0}$,
\[
\mathcal{R}(q, d, k)
=\frac{(q+1)q \cdots (q-(k-2))}{k(q+1)^k} + O\lp\frac{1}{q}\rp,
\]
where the implied constant depends only on $d$.
In particular, $\mathcal{R}\lp q,d\rp=\mathcal{K}(q+1)+O(1)$.
\end{conjecture2}

\begin{ratheuristic} \label{ratheuristic} If $k \in \{1, \dots, q\}$, and $d \in \Z^{\geq 0}$, then 
\[ \E \lc V_{q,k} W_{q,d,k} \rc=\E \lc V_{q,k} \rc \E \lc W_{q,d,k} \rc +O\lp q^{2d-2k} \rp. \]
Here, the implied constant depends only on $d$.
\end{ratheuristic}

\begin{mainthm2} \label{ratmainthm2}
If \hyperref[ratheuristic]{Heuristic ~\ref*{ratheuristic}} is true, then \hyperref[conjecture2]{Conjecture~\ref*{conjecture2}} is true.
\end{mainthm2}

\begin{proof}
Similar to the proof of \hyperref[mainthm1]{Theorem ~\ref*{mainthm1}}.
\end{proof}
%
%
%

\section*{Acknowledgments}

The authors would like to thank Ian Dinwoodie, Rafe Jones, and Christopher Kramer for their help and advice.

\bibliography{graphheuristics}
\bibliographystyle{amsalpha}

\end{document}